\documentclass{amsart}
\usepackage{amssymb,amsfonts,mathrsfs}
\usepackage{cite}

\newcommand{\bS}{\boldsymbol{S}}

\newcommand{\bcdot}{\boldsymbol{\cdot}}
\newcommand{\bkappa}{\boldsymbol{\kappa}}

\newtheorem{theorem}{Theorem}

\theoremstyle{remark}
\newtheorem{remark}{Remark}

\begin{document}

\title[Spectral finiteness property]{Finiteness Property of a Bounded Set
of Matrices with Uniformly Sub-Peripheral Spectrum}

\author{Xiongping Dai}
\address{Department of Mathematics, Nanjing University, Nanjing 210093,
People's Republic of China}
\email{xpdai@nju.edu.cn}

\author{Victor Kozyakin}
\address{Institute for Information Transmission Problems, Russian Academy
of Sciences, Bolshoj Karetny lane, 19.
127994, GSP-4, Moscow, Russia} \email{kozyakin@iitp.ru}

\thanks{X.~Dai was supported partly by National Natural Science
Foundation of China (Grant no. 11071112) and PAPD of Jiangsu Higher
Education Institutions. V.~Kozyakin was supported partly by the
Russian Foundation for Basic Research, project no. 10-01-93112.}


\subjclass[{2000}]{Primary 15B52; Secondary 15A60, 93D20, 65F15.}



\begin{abstract}
In the paper, a simple condition guaranteing the finiteness property, for a bounded set $\bS=\{S_k\}_{k\in K}$ of real or complex $d\times d$ matrices, is presented. It is
shown that existence of a sequence of matrix products
$\bS_{\!{\sigma(n_{\ell})}}$ of length $n_\ell$ for $\bS$ with $n_{\ell}\to\infty$ such that
the spectrum of each matrix $\bS_{\!{\sigma(n_{\ell})}}$ is uniformly
sub-peripheral and
\[
{\rho}({\bS}):=\sup_{n\ge1}\sup_{i_1,\dotsc,i_n\in K}\sqrt[n]{{\rho(S_{i_1}\dotsm S_{i_n})}}=\lim_{\ell\to+\infty}\sqrt[\uproot{3}n_{\ell}]{\rho(\bS_{\!{\sigma(n_{\ell})}})},
\]
guarantees the spectral finiteness property for $\bS$.
\end{abstract}

\maketitle
\section{Introduction}\label{sec1}%
In this paper, we prove the finite-step realizability of the
joint/generalized spectral radius for a bounded set of matrices with the
so-called uniformly sub-peripheral spectrum.
\subsection{Joint and generalized spectral radii}\label{sec1.1}
Throughout this paper, we let
\[
\bS=\{S_k\}_{k\in K},\quad \mathrm{card}(K)\ge2,
\]
be a bounded set of $d\times d$ matrices over the field
$\mathbb{F}=\mathbb{R},\mathbb{C}$ indexed by elements from some set $K$.  Let also $\|\cdot\|$ be a row-vector
norm on ${\mathbb{F}}^{1\times d}$ and also the induced matrix-norm on
$\mathbb{F}^{d\times d}$. Associate with any finite-length word
\[
\sigma= \{i_{1},\ldots,i_{n}\}\in
K^n:=\stackrel{n\textrm{-time}}{\overbrace{K\times\dotsm\times K}}
\]
the matrix $\bS_{\!\sigma}=S_{i_1}\dotsm S_{i_n}$, and define for
any integer $n\ge1$ two quantities
\[
\hat{\rho}_{n}({\bS})=\sup_{\sigma\in K^n}{\|\bS_{\!\sigma}\|}\quad
\textrm{and}\quad {\rho}_{n}({\bS})=\sup_{\sigma\in K^n}{
\rho(\bS_{\!\sigma})}.
\]
Here $\rho(A)$ stands for the usual spectral radius for an
arbitrary matrix $A\in\mathbb{F}^{d\times d}$. Then by the
sub-multiplicative property $\|AB\|\le\|A\|\cdot\|B\|$ for all
$A,B\in\mathbb{F}^{d\times d}$ there exists the limit
\[
\hat{\rho}({\bS})=\limsup_{n\to+\infty}\sqrt[n]{\hat{\rho}_{n}({\bS})}\quad
\bigl(~=\lim_{n\to+\infty}\sqrt[n]{\hat{\rho}_{n}({\bS})}=\inf_{n\ge1}\sqrt[n]{\hat{\rho}_n(\bS)}\bigr),
\]
which does not depend on the choice of the norm $\|\cdot\|$. This
limit was called by  Rota and Strang \cite{RotaStr:IM60} \emph{the
joint spectral radius} of the matrix set $\bS$. Analogously, there
exists the limit
\[
{\rho}({\bS})=
\limsup_{n\to+\infty}\sqrt[n]{{\rho}_{n}({\bS})}\quad\bigl(~=\sup_{n\ge1}\sqrt[n]{{\rho}_{n}({\bS})}\bigr),
\]
which was called by Daubechies and Lagarias \cite{DaubLag:LAA92}
\emph{the generalized spectral radius} of the matrix set $\bS$. As
is shown in \cite{BerWang:LAA92}, for finite matrix sets $\bS$ the
quantities $\hat{\rho}({\bS})$ and ${\rho}({\bS})$ coincide with
each other, and for any $n$ the following inequalities hold
\begin{equation}\label{E0}
\sqrt[n]{{\rho}_{n}({\bS})}\le {\rho}({\bS})=\hat{\rho}({\bS})\le
\sqrt[n]{\hat{\rho}_{n}({\bS})}
\end{equation}
which are useful for numerical computation of the joint spectral radius
$\hat{\rho}({\bS})$.
\subsection{Spectral finiteness property}\label{sec1.2}
In \cite{LagWang:LAA95} Lagarias and Wang conjectured that for finite sets $\bS$ the value ${\rho}(\bS)$ in fact coincides with
$\sqrt[n]{\rho(\bS_{\!\sigma})}$ for some $n$ and $\sigma\in
K^{n}$; that is to say, $\bS$ has the spectral \emph{finiteness
property}. If this \emph{Finiteness conjecture} is true, then the
problem of determining whether ${\rho}(\bS)<1$ is decidable. This
is because if ${\rho}(\bS)<1$, then there exists $n$ such that
${\rho}_{n}({\bS})<1$, whereas if $\hat{\rho}(\bS)\ge 1$, the
Finiteness conjecture implies that there exists $n$ such that
$\hat{\rho}_{n}(\bS)\ge1$. By checking both conditions for
increasing values of $n$, one of them will be eventually satisfied
and a decision will be made after a finite amount of computation.
Note that for a single matrix the problem is decidable. So, the
Finiteness conjecture has strong implications on the computation of
the joint/generalized spectral radius.

A simplest example of matrix sets having the finiteness property
are bounded sets of matrices consisting of upper (or lower)
triangular matrices. Another trivial example deliver bounded matrix
sets $\bS$ consisting of matrices $S\in\bS$ `isometric to a scalar
factor' in some row-vector norm $\|\cdot\|$ on $\mathbb{F}^{1\times
d}$, i.e., such that for any $x\in\mathbb{F}^{1\times d}$,
$\|xS\|=\lambda_{S}\|x\|$ with some constant $\lambda_{S}$. One
more example was given by Plischke and Wirth \cite{PW:LAA08} who
proved that irreducible\footnote{A matrix set $\bS$ is called
\emph{irreducible}, if the matrices from this set have no common
invariant subspaces except $\{\mathbf{0}\}$ and
${\mathbb{F}}^{1\times d}$. We notice that this is completely
different with the `irreducibility' of a Markov transition matrix
in probability theory.\label{foot3}} bounded `symmetric' matrix
sets\footnote{A matrix set $\bS$ is called \emph{symmetric} if
$S\in\bS$ implies that $S^{*}\in\bS$, where $S^{*}$ is a matrix
conjugate to $S$.} possess the finiteness property. Less trivial
examples were constructed by Omladi\v{c} and Radjavi in
\cite{OmRadj:LAA97}, where they showed that the finiteness property
holds for matrix sets $\bS$ for which the semigroup $\bS^+$ of all
the products of matrices from $\bS$ possesses the so-called
`sub-multiplicative spectral radius property', i.e.,
$\rho(FH)\le\rho(F)\cdot\rho(H)$ for all $F,H\in\bS^+$.

In \cite{Gurv:LAA95} Gurvits showed that, for real matrix sets
$\bS$, the Finiteness conjecture holds if there is a real polytope
extremal norm\footnote{A norm $\|\cdot\|$ is called \emph{extremal}
for the matrix set $\bS$ if  $\|S\|\le\rho(\bS)$ for all
$S\in\bS$}. In \cite{LagWang:LAA95}, Lagarias \& Wang proved a more
general result that Finiteness conjecture holds if there is a
piecewise real analytic extremal norm. At last, as showed Guglielmi
\textit{et al}. \cite{GWZ:SIAMJMA05}, for complex matrix sets
$\bS$, the Finiteness conjecture holds if there is a complex
polytope extremal norm. However, to make use of these results, one
needs to know whether a set $\bS$ admits an extremal norm or not.
It was shown, e.g., in \cite{Bar:ARC88,Koz:AiT90:6:e,Dai:JMAA11}
that bounded irreducible sets of matrices always admit extremal
norms, yet nothing proves that polytope or piecewise analytic
extremal norms are always possible, see, e.g., \cite{Theys:PhD05}
and references therein. In \cite{GWZ:SIAMJMA05} Guglielmi
\textit{et al.} conjectured that every non-defective\footnote{A
matrix set $\bS$ is called \emph{non-defective} if the semigroup
generated by the set $\rho(\bS)^{-1}\bS$ is bounded.} finite family
of complex matrices that possesses the finiteness property has a
complex polytope extremal norm. Unfortunately, later on this
conjecture was disproved by Jungers \& Protasov
\cite{JP:SIAMJMA09}.

Despite of the above examples in which the Finiteness conjecture holds,
the Finiteness conjecture is turned to be false in general. The first
counterexample to the Finiteness conjecture was given by Bousch and
Mairesse in \cite{BM:JAMS02}, and the corresponding proof was essentially
based on the analysis of the so-called topical maps and Sturmian measures.
Later on in \cite{BTV:MTNS02,BTV:SIAMJMA03} Blondel, Theys and Vladimirov
proposed another proof of the counterexample to the Finiteness Conjecture,
which extensively exploited combinatorial properties of permutations of
products of positive matrices. In the control theory, as well as in the
general theory of dynamical systems, the notion of generalized spectral
radius is used basically to describe the rate of growth or decrease of the
trajectories generated by matrix products. In connection with this,
Kozyakin in \cite{Koz:CDC05:e,Koz:INFOPROC06:e} presented one more proof
of the counterexample to the Finiteness conjecture fulfilled in the spirit
of the theory of dynamical systems. In this proof, the method of Barabanov
norms \cite{Bar:ARC88} was the key instrument in disproving the Finiteness
conjecture. The related constructions were essentially based on the study
of the geometrical properties of the unit balls of some specific Barabanov
norms and properties of discontinuous orientation preserving circle maps.

To appreciate the merits of the above mentioned disproofs of the
Finiteness conjecture let us point out that the key ideas
underlying all the proofs in
\cite{BM:JAMS02,BTV:SIAMJMA03,Koz:CDC05:e,Koz:INFOPROC06:e} were
based on the frequency properties of the Sturmian sequences. In
\cite{BM:JAMS02} such properties were formulated and investigated
in terms of the so-called Sturmian ergodic invariant measures on
the spaces of binary sequences. In
\cite{BTV:SIAMJMA03,Koz:CDC05:e,Koz:INFOPROC06:e}, the ergodic
theory formally was not mentioned. However, the usage of
combinatorial properties of Sturmian sequences in
\cite{BTV:SIAMJMA03} or of the fact that Sturmian sequences
naturally arise in symbolic description of trajectories of
(discontinuous) orientation preserving circle rotation maps in
\cite{Koz:CDC05:e,Koz:INFOPROC06:e} were essentially motivated
namely by ergodic properties of Sturmian sequences.

Unfortunately, all the disproofs
\cite{BM:JAMS02,BTV:SIAMJMA03,Koz:CDC05:e,Koz:INFOPROC06:e} of the
Finiteness conjecture were pure `existence' (or, sooner,
`non-existence') unconstructive results. Only recently, in
\cite{HMST:AdvMath11} Hare \textit{et al.} combined the approaches
developed in
\cite{BM:JAMS02,BTV:SIAMJMA03,Koz:CDC05:e,Koz:INFOPROC06:e} with
some rapidly-converging lower bounds for the joint spectral radius
based on the multiplicative ergodic theory obtained by Morris
\cite{Morris:ADVM10}, which allowed them to build explicitly the
set of matrices for which the Finiteness conjecture fails. Namely,
for the matrix set
\[
\bS({\alpha}):=
\left\{S_1=\left[\begin{matrix}1&1\\0&1\end{matrix}\right],
S_2=\alpha\left[\begin{matrix}1&0\\1&1\end{matrix}\right]\right\}
\]
they computed an explicit value of
\[
\alpha_{*} \simeq
0.749326546330367557943961948091344672091327370236064317358024\ldots
\]
such that $\bS({\alpha_{*}})$ does not satisfy the finiteness property. It
is still unknown whether $\alpha_{*}$ is rational or not.

So, ideas of ergodic theory are proved to be fruitful in disproving the
Finiteness conjecture
\cite{BM:JAMS02,BTV:SIAMJMA03,Koz:CDC05:e,Koz:INFOPROC06:e,Morris:ADVM10,HMST:AdvMath11}.
Further development of the ergodic theory approach to investigation of the
properties of the joint spectral radius was done by Dai \emph{et al.} in
\cite{DHX:ERA11,DHX:AUT11}. Based on the classic multiplicative ergodic
theorem and the semiuniform subadditive ergodic theorem, they showed in
particular that there always exists at least one ergodic Borel probability
measure on the one-sided symbolic space $\varSigma_{\!K}^+$ of all
one-sided infinite sequences $i_{\bcdot}\colon\mathbb{N}\to K$
such that the joint spectral radius of a finite set of square matrices
$\bS$ can be realized almost everywhere with respect to this measure
\cite{DHX:ERA11}.

Since the Finiteness conjecture was proved to be invalid generally,
serious efforts were undertaken by some investigators to find less general
classes of matrices for which the Finiteness conjecture still might be
valid. One of the most interesting such classes constitute matrices with
rational entries. In \cite{JB:LAA08,Jungers:09}, Jungers and Blondel
showed that the finiteness property holds for nonnegative rational
matrices if and only if it holds for pairs of binary matrices, i.e.,
matrices with the entries $\{0,1\}$. So they conjectured that pairs of
binary matrices always have the finiteness property. In support to this
conjecture they proved that the finiteness property holds for pairs of
$2\times 2$ binary matrices. They gave also a similar result for matrices
with negative entries. Namely, they proved that the finiteness property
holds for (general) rational matrices if and only if it holds for pairs of
sign-matrices, i.e., matrices with entries $\{-1,0,1\}$. More recently,
Cicone \textit{et al.} in \cite{CGSZ:LAA10} proved that the finiteness
property holds for pairs of $2\times 2$ sign-matrices; and Dai \textit{et
al.} in \cite{DHLX11} proved that for any pair $\bS=\{S_1,S_2\}\subset
\mathbb{R}^{d\times d}$, if one of $S_1,S_2$ has the rank $1$, then $\bS$
possesses the finiteness property.

The aim of this paper is to present yet another sufficient condition
enabling the finiteness property of a set of matrices.

\section{Pure peripheral spectrum and main statement}\label{sec2}
Recall that an eigenvalue $\lambda$ of a matrix
$A\in\mathbb{F}^{d\times d}$ is said to belong to the
\emph{peripheral spectrum} of $A$ if $|\lambda|=\rho(A)$. If
$|\lambda|=\rho(A)$ for all eigenvalues $\lambda$ of $A$, then we
say that $A$ has the \emph{pure peripheral spectrum}. For example,
any unitary matrix has the pure peripheral spectrum. Let us say
that a family of matrices $\bS_{\!\sigma}$ has the \emph{uniformly
sub-peripheral spectrum} if there exists a constant $\bkappa$ with
$0<\bkappa<1$ such that each eigenvalue $\lambda$ of
$\bS_{\!\sigma}$ satisfies
$\bkappa\rho(\bS_{\!\sigma})\le|\lambda|\le \rho(\bS_{\!\sigma})$.
Clearly, if the spectrum of every matrix $\bS_{\!\sigma}$ is pure
peripheral then the whole family of matrices $\bS_{\!\sigma}$ has a
uniformly sub-peripheral spectrum.

Now our main statement may be formulated as follows:

\begin{theorem}\label{T-main}
Let $\bS=\{S_k\}_{k\in K}\subset\mathbb{F}^{d\times d}$  be a bounded set
of matrices. If there exists a sequence of matrix products
$\bS_{\!{\sigma(n_{\ell})}}$ for $\bS$, where $\sigma(n_{\ell})\in
K^{n_\ell}$ and $n_{\ell}\to+\infty$, such that its spectrum is uniformly
sub-peripheral and
\begin{equation}\label{E1}
\rho(\bS)=\lim_{\ell\to+\infty}\sqrt[\uproot{3}n_{\ell}]{\rho(\bS_{\!{\sigma(n_{\ell})}})},
\end{equation}
then $\bS$ possesses the spectral finiteness property with $\rho(\bS)=
\sup_{k\in K}\rho(S_k)$.
\end{theorem}

In light of the counterexample of Hare \textit{et al.}
\cite{HMST:AdvMath11} where $\bS=\{S_1,S_2\}$ and the spectra of
both $S_1$ and $S_2$ are pure peripheral, with $\rho(S_1)=1$ and
$\rho(S_2)=\alpha_*$, our assumption that the matrix sequence
$\langle\bS_{\!{\sigma(n_{\!\ell})}}\rangle_{\ell=1}^{+\infty}$ has
the uniformly sub-peripheral spectrum is essential for the
statement of Theorem~\ref{T-main}.

Let us present one example in which the claim of Theorem~\ref{T-main} is
evident. If each element of the multiplicative semigroup
$\bS^+\subset\mathbb{F}^{d\times d}$, generated by $\bS$, has the pure
peripheral spectrum then
\[
\rho(AB)=\sqrt[d]{\det(AB)}
=\sqrt[d]{\det(A)}\cdot\sqrt[d]{\det(B)}=\rho(A)\cdot\rho(B),\quad
\forall~ A,B\in\bS^+.
\]
This implies that
\[
\rho(\bS)=\sup_{k\in K}\rho(S_k),
\]
and so in this case the set of matrices $\bS$ has the spectral
finiteness property.

Matrix multiplicative semigroups satisfying
$\rho(AB)=\rho(A)\cdot\rho(B)$ for any their members $A$ and $B$
are called \emph{semigroups with multiplicative spectral radius},
see, e.g. \cite{OmRadj:LAA97}. As is shown in
\cite[Theorem~2.5]{OmRadj:LAA97}, for any such irreducible
semigroup of matrices there exists a (vector) norm $\|\cdot\|$ in
which each matrix from the semigroup is a direct sum of isometry
(in the norm $\|\cdot\|$) and a nilpotent matrix. Nontrivial
examples of semigroups with multiplicative spectral radius can be
found in \cite{OmRadj:LAA97}.

Let us remark that under the conditions of Theorem~\ref{T-main} the
set of matrices
$\left\langle\bS_{\!{\sigma(n_{\ell})}}\right\rangle_{\ell=1}^{+\infty}$
does not need to be a semigroup and moreover this set in general
lacks the multiplicative spectral radius property. Still,
Theorem~\ref{T-main} is valid in this more restrictive, comparing
with Theorem~2.5 from \cite{OmRadj:LAA97}, situation, too.

\begin{proof}
Let
$\left\langle\bS_{\!{\sigma(n_{\ell})}}\right\rangle_{\ell=1}^{+\infty}$
be a sequence of matrix products for $\bS$ specified by the condition of
Theorem~\ref{T-main}. Then, since the family of matrix products
$\langle\bS_{\!{\sigma(n_{\ell})}}\rangle_{\ell=1}^{+\infty}$ has the
uniformly sub-peripheral spectrum, we have
\begin{align*}
\bkappa\rho(\bS_{\!{\sigma(n_{\ell})}})&\le\sqrt[d]{\det(S_{i_1(n_{\ell})}\dotsm
S_{i_{n_{\ell}}(n_{\ell})})}=
\sqrt[d]{\det(S_{i_1(n_{\ell})})\dotsm
\det(S_{i_{n_{\ell}}(n_{\ell})})}\\
&=\sqrt[d]{\det(S_{i_1(n_{\ell})})}\dotsm
\sqrt[d]{\det(S_{i_{n_{\ell}}(n_{\ell})})}\le
\rho(S_{i_1(n_{\ell})})\dotsm \rho(S_{i_{n_{\ell}}(n_{\ell})})\\
&\le\left(\sup_{k\in K}\rho(S_k)\right)^{n_{\ell}},
\end{align*}
where $\sigma(n_{\ell})=(i_1(n_{\ell}),\dotsc,i_{n_{\ell}}(n_{\ell}))$,
for some constant $0<\bkappa<1$. Together with \eqref{E1}, these latter
inequalities imply that
\begin{equation}\label{E2}
\rho(\bS)=
\lim_{\ell\to+\infty}\sqrt[\uproot{3}n_{\ell}]{\rho(\bS_{\!{\sigma(n_{\ell})}})}\le
\lim_{\ell\to+\infty}\sqrt[\uproot{3}n_{\ell}]{\bkappa^{-1}\left(\sup_{k\in
K}\rho(S_k)\right)^{n_{\ell}}}= \sup_{k\in K}\rho(S_k).
\end{equation}
On the other hand, by \eqref{E0} we have
\[
\rho(\bS)\ge \sup_{k\in K}\rho(S_k),
\]
which together with \eqref{E2} implies
\[
\rho(\bS)= \sup_{k\in K}\rho(S_k).
\]
Theorem~\ref{T-main} is thus proved.
\end{proof}

\begin{remark}\label{R1}
From Theorem~\ref{T-main} it follows that if a
matrix set $\bS$ does not possesses the finiteness property then for any
sequence $\langle\bS_{\!{\sigma(n_{\ell})}}\rangle$ satisfying condition
\eqref{E1} the minimal absolute value of the eigenvalues of
$\bS_{\!{\sigma(n_{\ell})}}$ divided by $\rho(\bS_{\!{\sigma(n_{\ell})}})$
tends to zero as $\ell\to+\infty$.
\end{remark}

Let us recall now from \cite{Wirth:LAA02} that \emph{the limit
semigroup} $\bS_{\!\infty}$ generated by the set of matrices $\bS$
is defined to be the set of all limit points for the matrix
sequences
$\left\langle\rho(\bS)^{-n_{\ell}}\bS_{\!{\sigma(n_{\ell})}}\right\rangle_{\ell=1}^{+\infty}$,
where $\sigma(n_{\ell})\in K^{n_\ell}$ and $n_{\ell}\to\infty$. As
is known \cite{Wirth:LAA02}, the limit semigroup is nonempty
bounded when the set of matrices $\bS$ is irreducible.

As a consequence of Theorem~\ref{T-main}, we can obtain a
sufficient condition for the finiteness property for an irreducible
$\bS$.

\begin{theorem}\label{thm2.3}
Let an irreducible bounded set of matrices $\bS=\{S_k\}_{k\in
K}\subset\mathbb{F}^{d\times d}$ do not possess the finiteness property.
Then any matrix $A\in\bS_{\!\infty}$ is degenerate, that is, $\det A=0$.
\end{theorem}

\begin{proof}
Since the bounded set of matrices $\bS$ is irreducible then
$0<\rho(\bS)<+\infty$, see, e.g., \cite{Wirth:LAA02}. So, without loss of
generality, it may be assumed that $\rho(\bS)=1$. Fix an arbitrary matrix
$A\in\bS_{\!\infty}$. Then there exists a sequence of finite-length words
$\langle\sigma(n_\ell)\rangle_{\ell=1}^{+\infty}$ with
$\sigma(n_{\ell})\in K^{n_\ell}$ and $n_{\ell}\to\infty$ as
$\ell\to+\infty$, such that
\[
A=\lim_{\ell\to+\infty}\bS_{\!{\sigma(n_{\ell})}}.
\]
If $A$ would be singular, then we could stop our proof here. Next, we
assume $\det A\not=0$ and then $\rho(A)>0$. Since
$\rho(\bS_{\!{\sigma(n_{\ell})}})$ converges to $\rho(A)$ as
$\ell\to+\infty$, it follows
\[
\sqrt[\uproot{2}n_\ell]{\rho(\bS_{\!{\sigma(n_{\ell})}})}\to1=\rho(\bS).
\]
So, condition \eqref{E1} holds for the sequence
$\langle\bS_{\!{\sigma(n_{\ell})}}\rangle$. Then denoting by
$\lambda_{\ell}$ an eigenvalue of $\bS_{\!{\sigma(n_{\ell})}}$ with the
smallest absolute value we get by Remark~\ref{R1} that
\begin{equation}\label{E3}
\lambda_{\ell}\to0\quad\mbox{as}\quad \ell\to\infty.
\end{equation}
Now, by \eqref{E0} all the other eigenvalues of the matrix
$\bS_{\!{\sigma(n_{\ell})}}$ have the absolute values do not exceeding
$1$. So,
\[
|\det\bS_{\!{\sigma(n_{\ell})}}|\le|\lambda_{\ell}|
\]
and by \eqref{E3} we obtain
\[
\det A=\lim_{\ell\to+\infty}\det\bS_{\!{\sigma(n_{\ell})}}=0,
\]
which is a contradiction to the assumption.

Due to arbitrariness of the matrix $A\in\bS_{\!\infty}$, the theorem is
thus proved.
\end{proof}

It is interesting to formulate Theorem~\ref{thm2.3} equivalently as
follows:

\begin{theorem}\label{cor2.4}
Let a bounded set of matrices $\bS=\{S_k\}_{k\in
K}\subset\mathbb{F}^{d\times d}$ be irreducible. If there exists a
nonsingular $A\in\bS_{\!\infty}$, then $\bS$ has the finiteness property.
\end{theorem}

Recall from \cite{Morris:LAA10} that an irreducible set of matrices $\bS$
is said to have \emph{rank one property} if every nonzero element of
$\bS_{\!\infty}$ has rank one. Then from
\cite[Corollary~1.6]{Morris:LAA10} it follows that for each $\mathrm{card}(K),d\ge2$
there exists an irreducible finite set of matrices
$\bS=\{S_k\}_{k\in K}\subset\mathbb{F}^{d\times d}$ which satisfies both the finiteness and the rank one properties; for example,
\[
\bS=\left\{\left[\begin{matrix}1&0\\0&\lambda\end{matrix}\right],\left[\begin{matrix}0&\lambda\\\lambda&0\end{matrix}\right]\right\},\quad
\textrm{where }0<|\lambda|<1,
\]
has the finiteness and rank-one properties \cite[Example
2]{Morris:LAA10}. Therefore, the condition presented in
Theorem~\ref{cor2.4} does not need to be necessary for the
finiteness property. However, there always exist open subsets of
irreducible pairs $\bS=\{S_1,S_2\}\subset\mathbb{R}^{2\times 2}$ in
which the rank one property does not hold \cite{Morris:LAA10} and
hence the finiteness property holds from Theorem~\ref{cor2.4}. In
addition, the counterexample $\bS(\alpha_*)$ of Hare \textit{et
al.} \cite{HMST:AdvMath11}, mentioned in Section~\ref{sec1.2}, has
the rank one property from Theorem~\ref{cor2.4}.
\section{A stability criterion from periodically switched
stability}\label{sec3}
Let us recall that a finite set
$\bS=\{S_k\}_{k\in K}\subset\mathbb{F}^{d\times d}$ is called
\emph{periodically switched stable} if $\rho(\bS_{\!\sigma})<1$ for all
$\sigma\in K^n$ and all $n\ge1$; see, e.g.,
\cite{SWMWK:SIAMREV07,DHX:AUT11}. The following question of substantial
importance was posed by E.\,S.~Pyatniski\v{i} in 1980s: when does
periodically switched stability imply the absolute stability for $\bS$?

Since the spectral finiteness property is equivalent to the absolute
stability of some periodically switched stable system, then from
Theorem~\ref{T-main} it follows immediately the following stability
criterion:

\begin{theorem}\label{thm3.1}
Let $\bS=\{S_k\}_{k\in K}\subset\mathbb{F}^{d\times d}$ be
periodically switched stable, where $K$ is a finite index set. If
there exists a sequence of words
$\langle\sigma(n_{\ell})\rangle_{\ell=1}^{+\infty}$ satisfying the
requirements of Theorem~\ref{T-main}, then $\bS$ is absolutely
stable; that is, $\|S_{i_1}\dotsm S_{i_n}\|\to0$ as $n\to+\infty$
for all one-sided infinite switching sequences
$i_{\bcdot}\colon\mathbb{N}\rightarrow K$.
\end{theorem}

So, in the situation of Theorem~\ref{T-main}, the stability of $\bS$ is
algorithmically decidable.
\section{Concluding remarks}\label{sec4}
In this paper, we have presented a short survey on the spectral finiteness
property for a finite set of $d\times d$ matrices. We have proved also
that if a bounded set $\bS$ of $d\times d$ matrices satisfies the
so-called uniformly sub-peripheral spectrum condition and an approximation
property of Lyapunov exponents, then $\bS$ possesses the spectral finiteness
property. This result has direct implication for the stability problem for
finite sets of matrices.

\bibliographystyle{elsarticle-num}
\bibliography{DaiKoz}
\end{document}